\documentclass[a4paper,twoside,11pt,reqno]{amsart}
\newcommand{\Ueberschrift}{Series of $p$-groups with Beauville structure}
\newcommand{\Kurztitel}{Series of $p$-groups with Beauville structure}
\usepackage[headings]{fullpage}
\usepackage{amsmath} \usepackage{amssymb} \usepackage{amsopn}
\usepackage{amsthm} \usepackage{amsfonts} \usepackage{mathrsfs} 
\usepackage[dvips]{graphicx} \usepackage[usenames]{color}
\usepackage[all]{xy}
\usepackage[OT2, T1]{fontenc}
\usepackage[latin1]{inputenc}  

\usepackage{dsfont}

\usepackage[pdfpagelabels, pdftex]{hyperref}
\hypersetup{
  pdftitle={},
  pdfauthor={},
  pdfsubject={},
  pdfkeywords={},
  colorlinks=true,    
  linkcolor=red,     
  citecolor=blue,     
  filecolor=blue,      
  urlcolor=blue,       
  breaklinks=true,
  bookmarksopen=true,
  bookmarksnumbered=true,
  pdfpagemode=UseOutlines,
  plainpages=false}

\newcommand{\bF}{{\mathbb F}}

\newcommand{\bN}{{\mathbb N}}

\newcommand{\bP}{{\mathbb P}}

\newcommand{\bZ}{{\mathbb Z}}

\newcommand{\dB}{{\mathcal B}}

\newcommand{\surj}{\twoheadrightarrow} 

\DeclareMathOperator{\id}{id}

\DeclareMathOperator{\Aut}{Aut}

\DeclareMathOperator{\GL}{GL}

\newcommand{\ep}{\varepsilon}
\newcommand{\ph}{\varphi}

\newcommand{\ab}{{\rm ab}}

\DeclareMathOperator{\ord}{{\rm ord}}

\newtheorem{thm}{Theorem}

\newtheorem{prop}[thm]{Proposition}
\newtheorem{lem}[thm]{Lemma}

\newtheorem{cor}[thm]{Corollary}

\newtheorem{defi}[thm]{Definition}

\newtheorem{rmk}[thm]{Remark}

\newenvironment{pro*}[1][Proof]{{\it{#1:}} }{}

\begin{document}

%
%
%

\hrule width\hsize

\vspace{1cm}

\title[\Kurztitel]{\Ueberschrift} 
\author{\sc Jakob Stix}
\address{Jakob Stix, Institut f\"ur Mathematik, Johann Wolfgang Goethe--Universit\"at Frankfurt, Ro\-bert-Mayer-Stra\ss e~6--8,
60325~Frankfurt am Main, Germany}
\email{stix@math.uni-frankfurt.de}


\author{Alina Vdovina}
\address{Alina Vdovina,  School of Mathematics and Statistics, Newcastle University, Newcastle upon Tyne, NE1 7RU, UK}
\email{alina.vdovina@ncl.ac.uk}
 
\thanks{The second author acknowledge  the hospitality and support provided by MPIM-Bonn.}
\date{July 19, 2015}

\maketitle

\begin{quotation} 
\noindent \small {\bf Abstract} --- 
For every $p\geq 2$ we show that each finite $p$-group with an unmixed Beauville structure is part of a surjective infinite projective system of finite $p$-groups with compatible unmixed Beauville structures. This leads to the new notion of  an unmixed topological Beauville structure on pro-finite groups.
We further construct for $p \geq 5$ a new explicit infinite series of non-abelian $p$-groups that allow unmixed Beauville structures. 
\end{quotation}

\setcounter{tocdepth}{2} {\scriptsize \tableofcontents}

\section{Introduction}

The datum of an unmixed Beauville structure on a finite group $G$ encodes two finite branched $G$-covers 
$C \to \bP^1$ and $D \to \bP^1$ with ramification only above $0$, $1$ and $\infty$, such that the diagonal 
$G$ action on $C \times D$ is fixed point free. The quotient $X =  (C \times D)/G$ belongs to an interesting 
class of (rigid) algebraic surfaces, the systematic study of which was initiated by 
Catanese \cite{catanese:isogtoproduct}. These surfaces that are isogeneous to a product of smooth curves are now known as Beauville surfaces if $p_g = 0$, and those have been classified by Catanese, Bauer and Grunewald \cite{bcg08}
and Frapporti \cite{frapporti}. In general, surfaces that are isogeneous to a product of smooth curves play a r\^ole in the geography of complex surfaces of general type. They provide lower bounds for the complexity of the set of connected components for the moduli spaces of such surfaces with prescribed Chern numbers.

Our main result will be the construction of new series of finite $p$-groups with unmixed Beauville structures. These examples will form projective systems 
$(G_i, \ph_{ij})$ indexed by a directed set $I$, typically $I=\bN$ with the natural archimedian ordering. The projection maps $\ph_{ij} : G_j \to G_i$ will be surjective and map the chosen Beauville structure for $G_j$ to the one for $G_i$. 

Compatible Beauville structures in a projective system lead to a projective system of finite covers of the corresponding 
Beauville surfaces with the tower being unramified if and only if the signature does not change in the tower.

\begin{defi}
An \textbf{(unmixed) Beauville structure} on a finite group $G$ consists of an ordered pair of triples $(x,y,z)$ and $(a,b,c)$ of group elements such that 
\begin{enumerate}
\item $xyz =1 = abc$, 
\item $G$ is generated by $\{x,y,z\}$ and by $\{a,b,c\}$, 
\item no non-trivial power of an element of $\{x,y,z\}$ is conjugate to a power of an element of $\{a,b,c\}$.
\end{enumerate}
The \textbf{signature} of the Beauville structure is the tuple of orders of the elements $x,y,z,a,b,c$ and the Beauville structure is \textbf{balanced} if these orders are constant. The individual triples $(x,y,z)$ and $(a,b,c)$ are refered to as \textbf{half} of a Beauville structure.
\end{defi}

A natural question asks which finite groups do have a Beauville structure?
Beauville's original example showed that $(\bZ/5\bZ)^2$ has a Beauville
structure. This was later generalized by Catanese to $(\bZ/n\bZ)^2$ with $n$ coprime to $6$.

Other important classes of groups are simple groups and $p$-groups. 
 Existence of Beauville structures was shown 
for every non-abelian finite simple group (except $A_5$) by Malle and Guralnick~\cite{malleguralnik} and 
independently Fairbairn, Magaard and Parker~\cite{fairbairnmagaardparker:simplebeauville}.
Much less is known about non-abelian $p$-groups: some
examples of Beauville structures on $p$-groups 
of the order up to $p^7$ were constructed
 by Barker, Boston and Fairbairn~\cite{bbf:smallpgroups},
 and with $p=2$ by Barker, Boston, Peyerimhoff and the second author \cite{bbpv:newbeuaville}, and generalized for an infinite  family  in \cite{bbpv:two-groups} by the same authors.
For a more detailed survey on Beauville surfaces and groups we refer to \cite{bgv} and especially to Jones \cite[\S7]{jones:survey} for Beauville structures on $p$-groups.

\smallskip

In this paper we show that infinite families of
$p$-groups with Beauville structures exist in abundance.          

\begin{thm} \label{thmABC:pro-system}
Every finite $p$-group with an unmixed Beauville structure is part of an infinite projective system of finite $p$-groups with compatible unmixed Beauville structures. One may further insist that the signature of the first half remains constant thoughout the projective system.
\end{thm}

The following corollary answers Problem 4.8 posed by Fairbairn in \cite{fairbairn} in the affirmative (the proposed Beauville structures are explicit as $p$-quotients of the respective triangle groups).

\begin{cor}
For any prime number $p \geq 2$ there are infinitely many $p$-groups with unmixed Beauville structures.
\end{cor}
\begin{proof}
By Theorem~\ref{thmABC:pro-system} it suffices to find a single $p$-group with a Beauville structure. For $p \geq 5$ the simplest examples are the abelian groups $\bZ/p^n\bZ \times \bZ/p^n\bZ$ as recalled in Theorem~\ref{thm:abelianBeauvillestructures} (and in this case the corollary is well known). For $p=2$ and $p=3$ we refer to 
\cite[Theorem~4.5]{fairbairn}  for examples of $p$-groups with a Beauville structure.
\end{proof}

While the construction of the projective system in Theorem~\ref{thmABC:pro-system} is not explicit, the second construction for 
Theorem~\ref{thmABC:uniform} exploits the structure of a met\-abelian uniform $p$-group and is completely explicit.

\begin{thm} \label{thmABC:uniform}
Let $p$ be a prime, $n,m \in \bN$ and $\lambda \in (\bZ/p^m\bZ)^\times$ with $\lambda^{p^n} \equiv 1 \mod p^m$. The semidirect product 
\[
\bZ/p^m\bZ \rtimes_\lambda \bZ/p^n\bZ
\]
with action $\bZ/p^n\bZ \to  \Aut(\bZ/p^m\bZ)=(\bZ/p^m\bZ)^\times$ sending $1 \mapsto \lambda$ admits an unmixed Beauville structure if and only if $p \geq 5$ and $n=m$. 

All Beauville structures  are balanced of constant signature $p^n$.
\end{thm}

The special case $\lambda = 1 + p$ of Theorem~\ref{thmABC:uniform} is contained in (and inspiration came from) 
\cite[Lemma~10]{bbf:smallpgroups}. An extension to $\lambda = 1 + p^r$ with $1 \leq r \leq n=m$ is contained in 
Jones \cite[\S7]{jones:survey}. 
%
%
Theorem~\ref{thmABC:pro-system} will be proven in Section~\S\ref{sec:pro}, while Section~\S\ref{sec:uniform} contains the proof of Theorem~\ref{thmABC:uniform}.

%
%
%
%

\section{Projective systems of Beauville structures} \label{sec:pro}

Versions of the following lifting result have also been observed for example in \cite[Lemma 17]{bbf:smallpgroups} , or 
\cite[Lemma 4.2]{fuertesjones}. The special case of $p$-groups considered here allows us to exploit the Frattini property to ensure the generating property. 

\subsection{Triangle group lifting} We start with an observation for general finite groups.

\begin{prop} \label{prop:trianglegrouplifting-non-p-group}
Let $\tilde{G} \surj G$ be a quotient map of finite groups such that $G$ admits a Beauville structure given by the tuples 
$(x,y,z)$ and $(a,b,c)$.
Assume that $(x,y,z)$ lifts to a triple $(\tilde{x},\tilde{y},\tilde{z})$ in $\tilde{G}$ with $\tilde{x}\tilde{y}\tilde{z} = 1$, generating $\tilde{G}$ and preserving the orders of the respective elements. 

If  $(a,b,c)$ lifts to a generating triple $(\tilde{a},\tilde{b},\tilde{c})$ in $\tilde{G}$ with product $1$ (not necessarily preserving
the orders), then $\tilde{G}$ admits a Beauville structure given by the tuples $(\tilde{x},\tilde{y},\tilde{z})$ and $(\tilde{a},\tilde{b},\tilde{c})$.
\end{prop}
\begin{proof}
If  $(\tilde{x},\tilde{y},\tilde{z})$ and $(\tilde{a},\tilde{b},\tilde{c})$ do not form a Beauville structure, then without loss of generality there is  $n,m \in \bN$ and $g \in \tilde{G}$ such that $1 \not= \tilde{x}^n = g(\tilde{a}^m)g^{-1}$. 
Projecting to $G$ and using that $(x,y,z)$ and $(a,b,c)$ is a Beauville structure, we find $x^n =1$ in $G$. But since $\tilde{x}$ and $x$ have the same orders, we deduce $\tilde{x}^n  =1$, a contradiction.
\end{proof}

We now specialise Proposition~\ref{prop:trianglegrouplifting-non-p-group} to $p$-groups. The defining property of the Frattini subgroup allows to weaken the assumptions. We recall that the \textbf{Frattini subgroup} of a finite group $G$ is the group 
\[
\Phi(G) = \bigcap_{H < G, \text{ maximal}} H.
\]
For a $p$-group $G$ the Frattini subgroup is the kernel of the quotient map to the maximal elementary abelian $p$-quotient
\[
\Phi(G)  = G^p \cdot [G,G] = \ker(G \surj G^\ab/p G^\ab).
\]
A set of elements  generates  $G$ if and only if their images in $G/\Phi(G)$ form a set of generators.

\begin{prop} \label{prop:trianglegrouplifting}
Let $\tilde{G} \surj G$ be a quotient map of finite $p$-groups such that $G$ admits a Beauville structure given by the tuples 
$(x,y,z)$ and $(a,b,c)$.

Assume that $(x,y,z)$ lifts to a triple $(\tilde{x},\tilde{y},\tilde{z})$ in $\tilde{G}$ with $\tilde{x}\tilde{y}\tilde{z} = 1$, generating $\tilde{G}$ and preserving the orders of the respective elements. Then the Beauville structure lifts to a Beauville structure $(\tilde{x},\tilde{y},\tilde{z})$ and $(\tilde{a},\tilde{b},\tilde{c})$ of $\tilde{G}$
with $\tilde{a},\tilde{b}$ arbitrary lifts of $a,b$ to $\tilde{G}$ and $\tilde{c} = (\tilde{a}\tilde{b})^{-1}$.
\end{prop}
\begin{proof}
Since  $G$ admits a Beauville structure, the group is not cyclic and hence $G/\Phi(G)$ is an $\bF_p$-vector space of dimension $2$. The same holds for $\tilde{G}$, hence  the map $\tilde{G}/\Phi(\tilde{G}) \surj G/\Phi(G)$ is an isomorphism. By the Frattini property, the generating sets for $G$ are lifts of generating sets of $G/\Phi(G)$, thus lifts of generators for $G$ also necessarily generate $\tilde{G}$: the triple $(\tilde{a},\tilde{b},\tilde{c})$ also generates $\tilde{G}$.
Now the proof follows from Proposition~\ref{prop:trianglegrouplifting-non-p-group}.
\end{proof}

The (strict) triangle group of signature $p^m$, $p^n$, $p^r$ is the group 
\[
\Delta_{m,n,r} = \langle X,Y,Z | XYZ =1, X^{p^m} = Y^{p^n} = Z^{p^r} =1\rangle
\]
that we consider as a group 
with distinguished generators $X$, $Y$ and $Z$.

\begin{rmk} 
We refer to lifting of Beauville structures as in Proposition~\ref{prop:trianglegrouplifting} as \textbf{triangle group liftings} since we view both $G$ and $\tilde{G}$ via the datum of $(\tilde{x},\tilde{y},\tilde{z})$ as quotients of a triangle group
\[
\Delta_{m,n,r} \surj \tilde{G} \surj G.
\]
with $\ord(x) = p^m$, $\ord(y) = p^n$ and $\ord(z) = p^r$.
\end{rmk}

Recall that a \textbf{pro-finite} power of an element $g$ in a pro-finite group $G$ is any element of the pro-cyclic subgroup topologically generated by $g$ in $G$.

\begin{defi}
An \textbf{(unmixed topological) Beauville structure} on a pro-finite group $G$ 
consists of an ordered pair of triples $(x,y,z)$ and $(a,b,c)$ of group elements such that 
\begin{enumerate}
\item $xyz = 1 = abc$,
\item $G$ is topologically generated by $\{x,y,z\}$ and by $\{a,b,c\}$, 
\item no non-trivial (pro-finite) power of an element of $\{x,y,z\}$ is conjugate to a (pro-finite) power of an element of $\{a,b,c\}$.
\end{enumerate}
\end{defi}

If $G = \varprojlim_i G_i$ is a description as a surjective projective system of finite groups, then a compatible system of Beauville structures on the $G_i$ will describe a topological Beauville structure on $G$. The converse is not true in general: the image in $G_i$ of a  topological Beauville structure $(X,Y,Z)$ and $(A,B,C)$ on $G$ may fail to be a Beauville structure. If the orders of $X$, $Y$ and $Z$ in $G$ are finite, then
\[
M_{X,Y,Z}^0 := \bigcup_{n \in \bN, \ g \in G} g\{X^n,Y^n,Z^n\}g^{-1} \setminus \{1\}
\]
and also $M_{A,B,C} = \langle A \rangle \cup \langle B \rangle \cup \langle C \rangle$, the union of the subgroups generated topologically by $A$, $B$, and $C$ respectively, form compact subsets of $G$. Hence, the sets $M_{X,Y,Z}^0$ and $M_{A,B,C}$ are disjoint in $G$ if and only if there is a finite quotient $G \surj G_i$ such that the images there are disjoint. 
Then for all finer indices $j \to i$ the images of $(X,Y,Z)$ and $(A,B,C)$ in $G_j$ form indeed a compatible system of Beauville structures.

\smallskip

Let $(\Delta_{m,n,r})^{\wedge p}$ be the pro-$p$ completion  of $\Delta_{m,n,r}$. 

\begin{cor} \label{cor:topologicalBeauville}
Let  $(x,y,z)$ and $(a,b,c)$ be a Beauville structure on a finite $p$-group $G$ with orders $\ord(x) = p^m$, $\ord(y) = p^n$ and 
$\ord(z) = p^r$.  Then there is a Beauville structure on $(\Delta_{m,n,r})^{\wedge p}$ of the form  $(X,Y,Z)$ and $(A,B,C)$ such that the triples $(x,y,z)$ and $(a,b,c)$ are their image via a unique surjection $(\Delta_{m,n,r})^{\wedge p} \surj G$
\end{cor}
\begin{proof}
The assignment $X,Y,Z \mapsto x,y,z$ determines uniquely a surjection 
\[
(\Delta_{m,n,r})^{\wedge p} \surj G.
\] 
We write $(\Delta_{m,n,r})^{\wedge p} = \varprojlim_i \Delta_i$ as a projective limit of finite groups and index system $(\bN,<)$. We assume $\Delta_0 = G$. Let $(x_i,y_i,z_i)$ be the image of $(X,Y,Z)$. 

Now set $\dB_0 = \{(x,y,z) \text{ and } (a,b,c)\}$ for the set containing only the  initial Beauville structure. 
By Proposition~\ref{prop:trianglegrouplifting} the set $\dB_{i+1}$ of Beauville structures on  $\Delta_{i+1}$ of the form $(x_{i+1},y_{i+1},z_{i+1})$ and $(a',b',c')$ lifting a Beauville structure in $\dB_i$ is non-empty and anyway finite. The sets $\dB_i$ form a projective system and a standard compactness argument shows that $\varprojlim_i \dB_i$ is non-empty as well. Moreover, the inductive construction of the $\dB_i$ shows that $\varprojlim_i \dB_i \to \dB_0$ is surjective. Any element in  $\varprojlim_i \dB_i$ is described by two tuples in  $(\Delta_{m,n,r})^{\wedge p}$, namely the standard generators $(X,Y,Z) = \varprojlim_i (x_i,y_i,z_i)$ and some $(A,B,C)$ that form a topological Beauville structure in  $(\Delta_{m,n,r})^{\wedge p}$ mapping to the given one on for $G = \Delta_0$ as claimed.
\end{proof}

\begin{proof}[Proof of Theorem~\ref{thmABC:pro-system}]
The topological Beauville structure of Corollary~\ref{cor:topologicalBeauville} was constructed as a compatible family of Beauville structures on the projective system of the $\Delta_i$, and moreover so that the signature of the first half of the Beauville structure stays constant. 

It remains to show that $(\Delta_{m,n,r})^{\wedge p}$ is infinite. But we work under the assumption that there is a $p$-group $G$ with a Beauville structure where the orders of the first half are $p^m$, $p^n$ and $p^r$, hence the corresponding quotient $\pi : \Delta_{m,n,r} \surj G$ is smooth, i.e., $\ker(\pi)$ is the fundamental group of a compact Riemann surface of genus $g$ with 
\[
2g - 2 = \#G \cdot \big(1 - \frac{1}{p^n} - \frac{1}{p^m} - \frac{1}{p^r}\big)
\] 
by the Riemann-Hurwitz formula. By Proposition~3.2 of \cite{bcg:withoutreal}
the signature $(p^n,p^m,p^r)$ is hyperbolic (for Theorem~\ref{thmABC:pro-system} it suffices that the signature is different from $(2,2,2^n)$ when $G$ is easily seen to be a dihedral group and thus not admitting a Beauville structure). It follows that $g \geq 2$ and $\ker(\pi)$ is a non-abelian surface group. Therefore its pro-$p$ completion $\ker(\pi)^{\wedge p}$ is infinite and therefore  also  $(\Delta_{m,n,r})^{\wedge p}$ is infinite. This completes the proof of Theorem~\ref{thmABC:pro-system}.
\end{proof}

\subsection{Some Beauville structures on \texorpdfstring{$p$}{p}-groups} 
The use of Theorem~\ref{thmABC:pro-system} is limited to the extent that the construction is not explicit (although quite flexible), and also subject to knowing $p$-groups with Beauville structures to start with. For completeness here are some examples.

\begin{thm}[Bauer, Catanese, Grunewald \cite{bcg:withoutreal} Thm.~3.4] \label{thm:abelianBeauvillestructures}
An abelian \\ group $G$ admits a Beauville structure if and only if $G \simeq \bZ/n\bZ \times \bZ/n\bZ$ and $(6,n) = 1$.
\end{thm}

The following non-abelian examples of balanced  signature $p$ generalize examples from 
Barker, Boston and Fairbairn~\cite{bbf:smallpgroups}.

\begin{prop} \label{prop:subofgeneralizedheisenberg}
Let $p \geq 5$ be a prime number. 
A $p$-group $G$ that can be embedded in $\GL_p(\bF_p)$ admits a Beauville structure if and only if $\dim_{\bF_p} G/\Phi(G) = 2$, that is, if and only if $G$ is generated by $2$ elements but is not cyclic.
\end{prop}
\begin{proof}
If $G$ admits a Beauville structure, then $G$ is not cyclic and $G/\Phi(G)$ is of dimension $2$.
For the converse direction we consider the Jordan normal form of $g \in G$ that has Jordan blocks of length at most $p$, hence all nontrivial elements in $G$ are of order $p$. Therefore we can lift a Beauville structure along 
$G \surj G/\Phi(G) \cong \bZ/p\bZ \times \bZ/p\bZ$
by triangle group lifting as in Proposition~\ref{prop:trianglegrouplifting}. The Frattini quotient admits a 
Beauville structure by  \cite[Theorem 3.4]{bcg:withoutreal}  as recalled above.
\end{proof}

\section{An example of a uniform group with a Beauville structure}
\label{sec:uniform}

By $\bZ_p$ we denote the $p$-adic integers. 
For an element $\lambda \in (\bZ/p^m\bZ)^\times$ of order dividing $p^n$ we consider the semi-direct product 
\[
G = \bZ/p^m\bZ \rtimes_\lambda \bZ/p^n\bZ
\]
with action $\bZ/p^n\bZ \to  \Aut(\bZ/p^m\bZ)$ sending $1 \mapsto \lambda$. We choose a lift denoted also by $\lambda \in \bZ_p^\times$ and consider the analoguous pro-finite semi-direct product
\[
\Gamma = \bZ_p \rtimes_\lambda \bZ_p.
\]

\begin{lem} \label{lem:magicofuniformgroups}
The subsets $\Gamma_r := p^r \bZ_p \rtimes_\lambda p^r \bZ_p \subseteq \Gamma$ (where $p^r$ acts by $\lambda^{p^r}$) form an exhaustive sequence of normal subgroups.  More precisely, the following holds.
\begin{enumerate}
\item  If $p$ is odd, or if $p=2$ and $\lambda \equiv 1 \mod 4$, then 
$\Gamma_r = \Gamma^{p^r}$ is the set of $p^r$-th powers for all $r \geq 0$. Moreover, the map $\gamma \mapsto P_s(\gamma) :=\gamma^{p^s}$ induces for all $r,s \geq 0$  a group isomorphism
\[
P_s: \Gamma_r/\Gamma_{r+1} \xrightarrow{\sim} \Gamma_{r+s}/\Gamma_{r+s+1}
\]
of groups isomorphic to $\bZ/p\bZ \times \bZ/p\bZ$.
\item If $p=2$ and  $\lambda \equiv -1 \mod 4$, then for $r \geq 1$ we have $\Gamma^{2^r} \subseteq \Gamma_r = \Gamma_1^{2^{r-1}}$. Moreover, the map $\gamma \mapsto P_s(\gamma) := \gamma^{p^s}$ induces for all $s \geq 0$ and $r \geq 1$ a group isomorphism
\[
P_s: \Gamma_r/\Gamma_{r+1} \xrightarrow{\sim} \Gamma_{r+s}/\Gamma_{r+s+1}
\]
of groups isomorphic to $\bZ/p\bZ \times \bZ/p\bZ$.
\end{enumerate}
\end{lem}
\begin{proof}
Since $\lambda$ has $p$-primary order, we must have $\lambda \equiv 1$ modulo $p$. It follows by induction that $\lambda^{p^r} \equiv 1$ modulo $p^{r+1}$. This shows that $\Gamma_r$ is indeed a normal subgroup, namely the kernel of the natural map $\Gamma \surj \bZ/p^r\bZ \rtimes_\lambda \bZ/p^r\bZ$.

Let $\mu \in \bZ_p$ be such that  $\lambda \equiv  1 + \mu p$ modulo $p^2$  and so $\lambda^x  \equiv 1 + p\mu x$ modulo $p^2$. 
The $p$-th power map $\Gamma \to \Gamma$ is then the $p$-adic analytic map 
\[
(a,x) \mapsto p \cdot (a,x) = (a \cdot (1 + \lambda^x + \lambda^{2x} + \ldots + \lambda^{(p-1)x}),px)  = (pa\cdot \ep_\lambda(x), px)
\]
where $\ep_\lambda(x)$ is a $p$-adic analytic function $\ep_\lambda : \bZ_p \to \bZ_p$ computed as
\[
\ep_\lambda(x) = \frac{1}{p} \cdot \sum_{i=0}^{p-1} \lambda^{ix} \equiv 1 + \mu x \binom{p}{2}   \equiv  
\left\{
\begin{array}{cl}
1 & \text{if  $p>2$ } \\
1+\mu x & \text{ if $p=2$ }
\end{array}
\right.
\mod p.
\]
If $p$ is odd, or if $p=2$ and $2 \mid \mu x$, then $\ep_\lambda(x)$ takes values in  $1+p\bZ_p$.
Now both assertions follow at once. 
\end{proof}

\begin{proof}[Proof of Theorem~\ref{thmABC:uniform}]
We first assume that $\lambda \equiv 1 \mod 4$ if $p=2$. 

We make use of the quotient map $\Gamma \surj G$ sending $(a,x)$ to $(a,x)$ that exists by the assumptions on $\lambda$. Let $r = \min\{n,m\}$ and assume first that $n\not=m$. Then the image $G_r \subseteq G$ of 
\[
\Gamma_r \to \Gamma \to G
\]
is a non-trivial cyclic subgroup. Any generating system $(g_1,\ldots,g_s)$ of $G$ maps to a generating system of $G/\Phi(G)$. By lifting to $\Gamma/\Gamma_1 \surj G/\Phi(G)$ we control that the powers $g_i^{p^r} \in G_r$ generate $G_r$, hence at least one of them is non-trivial. It follows that no two generating systems of $G$ can have disjoint sets of non-trivial powers. This shows that in order for $G$ to admit a Beauville structure we must have $n=m$.

From now on we assume that $n=m$. Let $(x,y,z)$ be a generating triple in $G$ with $xyz = 1$. Then image triple $(\bar x, \bar y, \bar z)$ in $G/\Phi(G) = \Gamma/\Gamma_1$ generates this $2$-dimensional $\bF_p$ vector space, hence all three elements are nontrivial. Since by Lemma~\ref{lem:magicofuniformgroups} then also their $p^{n-1}$-th powers are non-trivial, we deduce that all three have order $p^n$.   This proves that any potential Beauville structure must be balanced of constant signature $p^n$. 

Moreover, let $(x,y,z)$ and $(a,b,c)$ be a Beauville structure on $G$, we see that the $p^{n-1}$-th powers of these elements must yield a Beauville structure on $\Gamma_{n-1}/\Gamma_n \simeq \bZ/p\bZ \times \bZ/p\bZ$. 
In view of Theorem~\ref{thm:abelianBeauvillestructures} there cannot be a Beauville structure if $p < 5$. It remains to construct a Beauville structure if $p \geq 5$, and to exclude the case $p=2$ and $\lambda \equiv -1 \mod 4$.

\smallskip

We assume now that $p=2$ and $\lambda \equiv -1 \mod 4$, and we assume that $G$ has a Beauville structure $(x,y,z)$ and $(a,b,c)$. Note that in this case we do not yet know that $n=m$. Then the Frattini quotient $G/ \Phi(G)$ must be of order $4$ with both $(x,y,z)$ and $(a,b,c)$ mapping bijectively to the nontrivial elements. We set $G_r$ for the image of the map $\Gamma_r \to G$, and may assume without loss of generality that $x = a \cdot \delta$ with $\delta \in G_1$. 

More generally, if $g,h \in G_r$ differ by $\ep = h^{-1} g  \in G_{r+1}$, then 
\[
g^2 = (h \ep)^2 = h^2 (h^{-1}\ep h \ep^{-1}) \ep^2
\]
and so $g^2,h^2 \in G_{r+1}$ differ by $(h^{-1}\ep h \ep^{-1}) \ep^2 \in G_{r+2}$, because $G_r/G_{r+1}$ is central in $G/G_{r+1}$. By induction we find
\[
x^{2^{r-1}} \equiv a^{2^{r-1}} \mod G_{r}
\]
which contradicts the properties of a Beauville structure for $s = \max\{n,m\}$, when $G_s = 0$, if at least one of the elements $x,y,z,a,b,c$ has order $2^s$. 

Working modulo $4$ and using $\lambda \equiv -1 \mod 4$, we see that squares are either $(2,0)$ or $(0,2)$ modulo $G_2$, and among the elements $x,y,z$ both possibilities occur. By induction for any $g \in G$ we have, if $g^2 \equiv (2,0)$ that $g^{2^{r-1}} \equiv (2^{r-1},0) \mod G_r$, or if $g^2 \equiv (0,2)$ that $g^{2^{r-1}} \equiv (0,2^{r-1}) \mod G_r$. This shows that at least one element of every generating set has order $2^s$ with $s = \max\{n,m\}$. This excludes the case $p=2$ from allowing Beauville structures.

\smallskip

Let now $p \geq 5$ and  $(x,y,z)$ and $(a,b,c)$ be triples with product equal to one that map to a Beauville structure on $G/\Phi(G)$, in particular to generating triples of $G/\Phi(G)$. Then both triples $(x,y,z)$ and $(a,b,c)$ generate $G$ and it remains to show that there are no non-trivial conjugate powers. We argue by contradiction and assume without loss of generality that 
\[
0 \not= x^k = ga^lg^{-1}
\]
with $k,l \in \bN$ and $g \in G$. Let $r$ be maximal such that $x^k \in \Gamma_r/\Gamma_n \subseteq G$. By 
Lemma~\ref{lem:magicofuniformgroups} it follows that $p^r$ is the precise power dividing both $k$ and $l$. Let $k = p^r k_0$ and $l = p^r l_0$. Exploiting the isomorphism 
\[
P_r: G/\Phi(G) = \Gamma/\Gamma_1 \xrightarrow{\sim} \Gamma_r/\Gamma_{r+1}
\]
of raising to the $p^{r}$-th power of Lemma~\ref{lem:magicofuniformgroups}  we obtain  the equality 
\[
P_r(x^{k_0}) \equiv x^k \equiv g a^l g^{-1} \equiv g P_r(a^{l_0}) g^{-1} \equiv P_r(g a^{l_0} g^{-1})  \mod \Gamma_{r+1}
\]
of nontrivial elements in $\Gamma_r/\Gamma_{r+1}$ by the choice of $r$.
Since $P_r$ is an isomorphism, we conclude
\[
0 \not= x^{k_0} = ga^{l_0} g^{-1} \in G/\Phi(G).
\]
This contradicts our initial choice that $(x,y,z)$ and $(a,b,c)$ were a lift of a Beauville structure on $G/\Phi(G)$. This completes the proof.
\end{proof}

\begin{rmk}
(1) 
The proof of Theorem~\ref{thmABC:uniform} given above shows more. Every Beauville structure maps to a Beauville structure of 
$G/\Phi(G)$. And, conversely, any lift of a Beauville structure of the Frattini quotient yields a Beauville structure on $G$.

(2) 
The groups $\bZ/p^n\bZ \rtimes_\lambda \bZ/p^n\bZ$ are classified for $p \not= 2$ up to isomorphism as follows. The group order is $p^{2n}$, hence $p$ and $n$ are unique. The group is abelian if $\lambda \equiv 1 \mod p^n$ and otherwise the maximal abelian quotient 
\[
\bZ_p/(p^n, \lambda-1)\bZ_p \times \bZ/p^n\bZ
\]
has order $p^{n+s}$ where $0 < s = v_p(\lambda-1) < n$ is the $p$-adic valuation of any lift of $\lambda-1$ to $\bZ_p$. It follows that $n > r = n-s>0$ is an invariant of the group in the non-abelian case, namely the order of $\lambda \in \Aut(\bZ/p^n\bZ)$. 

The tuple $(p,n,r)$ is a complete set of invariants, since for $\lambda'$ leading to the same invariants there is an automorphism $\ph$ of $\bZ/p^n\bZ$ such that 
\[
(\id, \ph) : \bZ/p^n\bZ \rtimes_\lambda \bZ/p^n\bZ \xrightarrow{\sim} \bZ/p^n\bZ \rtimes_{\lambda'} \bZ/p^n\bZ
\]
yields an isomorphism. 

This classification shows that we have described plenty of Beauville structures on a family of non-abelian $p$-groups parametrized by integers $n > r > 0$.
\end{rmk}


\end{document}